\documentclass{amsart}
\usepackage{amsmath}
\usepackage{amssymb}
\usepackage[all]{xy}
\usepackage{graphicx}
\usepackage{mathrsfs}
\usepackage{stmaryrd}

\numberwithin{equation}{section}

\usepackage{url}

\usepackage[latin1]{inputenc}
\usepackage{xspace,amssymb,amsfonts,euscript}
\usepackage{amsthm,amsmath}
\usepackage{palatino}
\usepackage{euscript}
\input xy \xyoption {all}

\usepackage{tikz}
\usetikzlibrary{patterns,snakes}

\RequirePackage{color}
\definecolor{myred}{rgb}{0.75,0,0}
\definecolor{mygreen}{rgb}{0,0.5,0}
\definecolor{myblue}{rgb}{0,0,0.65}

\RequirePackage{ifpdf}
\ifpdf
  \IfFileExists{pdfsync.sty}{\RequirePackage{pdfsync}}{}
  \RequirePackage[pdftex,
   colorlinks = true,
   urlcolor = myblue, 
   citecolor = mygreen, 
   linkcolor = myred, 
   pagebackref,
   bookmarksopen=true]{hyperref}
\else
  \RequirePackage[hypertex]{hyperref}
\fi

\RequirePackage{ae, aecompl, aeguill} 

    \def\CM{{\mathbb{C}}}

    \def\QM{{\mathbb{Q}}}

    \def\ZM{{\mathbb{Z}}}

    \def\AC{{\mathcal{A}}}
    \def\BC{{\mathcal{B}}}
    \def\CC{{\mathcal{C}}}
    \def\DC{{\mathcal{D}}}

    \def\HC{{\mathcal{H}}}

    \def\NC{{\mathcal{N}}}
    \def\OC{{\mathcal{O}}}

    \def\SC{{\mathcal{S}}}

\def\a{\alpha}
\def\b{\beta}
\def\g{\gamma}
\def\G{\Gamma}
\def\d{\delta}

\def\e{\varepsilon}

\def\l{\lambda}
\def\L{\Lambda}

\def\th{\theta}

\def\z{\zeta}

\newcommand{\nc}{\newcommand} \newcommand{\renc}{\renewcommand}

\newcommand{\rdots}{\mathinner{ \mkern1mu\raise1pt\hbox{.}
    \mkern2mu\raise4pt\hbox{.}
    \mkern2mu\raise7pt\vbox{\kern7pt\hbox{.}}\mkern1mu}}

\def\un{\underline}

\def\p{{}^p}

\def\to{\rightarrow}

\def\longto{\longrightarrow}

\def\onto{\twoheadrightarrow}
\nc{\triright}{\stackrel{[1]}{\to}}
\nc{\longtriright}{\stackrel{[1]}{\longto}}

\nc{\Hb}{H^\bullet}

\nc{\Br}{\mathcal{B}}
\nc{\HotRR}{{}_R\mathcal{K}_R}
\nc{\HotR}{\mathcal{K}_R}
\nc{\excise}[1]{}
\nc{\defect}{\text{df}}
\nc{\h}[1]{\underline{H}_{#1}}

\nc{\Ga}{\mathbb{G}_a} 
\nc{\Gm}{\mathbb{G}_m} 

\nc{\Perv}{{\mathbf{P}}}

\nc{\IH}{{\mathrm{IH}}}

\nc{\ic}{\mathbf{IC}}

\nc{\gl}{{\mathfrak{gl}}}
\renc{\sl}{{\mathfrak{sl}}}
\renc{\sp}{{\mathfrak{sp}}}

\renc{\Im}{\textrm{Im}}

\nc{\HBM}{H^{BM}}

 \DeclareMathOperator{\Hom}{Hom}
 \DeclareMathOperator{\ch}{ch}
\DeclareMathOperator{\End}{End} 

\DeclareMathOperator{\Rep}{Rep}

\DeclareMathOperator{\id}{id}

\newtheorem{thm}{Theorem}[section]
\newtheorem{lem}[thm]{Lemma}

\newtheorem{prop}[thm]{Proposition}
\newtheorem{cor}[thm]{Corollary}

\theoremstyle{definition}

\newtheorem{ex}[thm]{Example}

\theoremstyle{remark}
\newtheorem{remark}[thm]{Remark}

\nc{\simto}{\stackrel{\sim}{\to}}

\DeclareMathOperator{\st}{st}

\newcommand{\Z}{\mathbb{Z}}

\renewcommand{\>}{\right\rangle}

\newcommand{\dd}{\partial}

\newcommand{\mattwos}[4]
{\bigl( \begin{smallmatrix}
                        #1  & #2   \\
                        #3 &  #4
\end{smallmatrix} \bigr)
}

\newcommand{\mattwo}[4]
{\left(\begin{array}{cc}
                        #1  & #2   \\
                        #3 &  #4
                          \end{array}\right) }

\newcommand{\SL}{\operatorname{SL}}

\newcommand{\figref}[1]{\hyperref[#1]{Figure \ref{#1}}}
\newcommand{\lemref}[1]{\hyperref[#1]{Lemma \ref{#1}}}
\newcommand{\thmref}[1]{\hyperref[#1]{Theorem \ref{#1}}}
\newcommand{\conjref}[1]{\hyperref[#1]{Conjecture \ref{#1}}}
\newcommand{\propref}[1]{\hyperref[#1]{Proposition \ref{#1}}}
\newcommand{\corref}[1]{\hyperref[#1]{Corollary \ref{#1}}}
\newcommand{\defref}[1]{\hyperref[#1]{Definition \ref{#1}}}
\newcommand{\rmkref}[1]{\hyperref[#1]{Remark \ref{#1}}}
\newcommand{\qref}[1]{\hyperref[#1]{Question \ref{#1}}}
\newcommand{\secref}[1]{\hyperref[#1]{\S\ref{#1}}}
\newcommand{\appref}[1]{\hyperref[#1]{Appendix \ref{#1}}}

\newcommand{\fc}{\frak{c}}
\nc{\St}{\mathrm{st}}
\nc{\df}{\mathrm{df}}

\title{Schubert calculus and torsion explosion}

\author[Geordie Williamson]{Geordie Williamson \\
with a joint appendix with \\
Alex Kontorovich and Peter J. McNamara }

\address{Rutgers University, New Brunswick, NJ}
\email{alex.kontorovich@rutgers.edu}

\address{University of Queensland, Brisbane, QLD, Australia.}
\email{p.mcnamara@uq.edu.au}

\address{Max-Planck-Institut f\"ur Mathematik,
Vivatsgasse 7, 53111,  Bonn, Germany.}
\email{geordie@mpim-bonn.mpg.de}

\begin{document}


\begin{abstract} We observe that certain numbers occurring in Schubert
  calculus for $\SL_n$ also occur as entries in intersection
  forms controlling decompositions of Soergel bimodules in higher rank. These numbers grow exponentially. This observation
  gives many counter-examples to the expected bounds in Lusztig's conjecture on
  the characters of simple rational modules for $\SL_n$ over fields of
  positive characteristic. Our examples also give
  counter-examples to the James conjecture on decomposition numbers
  for symmetric groups.
\end{abstract}

\maketitle

\begin{center}
  \emph{Dedicated to Meg and Gong.}
\end{center}

\section{Introduction}

Let $G$ be a connected algebraic group over an algebraically closed
field. A basic question in representation theory asks for the
dimensions and characters of the simple rational
$G$-modules. Structure theory of algebraic groups allows one to assume
that $G$ is reductive. If the ground field is of characteristic zero,
then the theory runs parallel to the well-understood theory for compact Lie groups. In
positive characteristic $p$, Steinberg's
tensor product theorem, the linkage principle and Jantzen's
translation principle reduce this to a question about finitely many
modules which occur in the same block as the trivial module (the
``principal block''). For these modules Lusztig has proposed a
conjecture if $p > h$,
 where $h$ denotes the Coxeter number of the
root system of $G$ \cite{L}.\footnote{Lusztig first proposed his conjecture
  under a restriction equivalent to $p \ge 2h-3$ (see \cite[\S 4]{JaCF} and \cite[\S 8.22]{Ja} for a discussion). 
Kato \cite[\S 5]{Kato} proved that if Lusztig's conjecture holds for
restricted weights then it holds for all weights in
the Jantzen region (Lusztig's original formulation). Since Kato's
work $p > h$ has been widely regarded as a realistic bound \cite[\S
  4]{JaCF}.}
He conjectures an expression for the characters of the simple
modules in terms of affine Kazhdan-Lusztig polynomials and the (known)
characters of standard modules.

Lusztig's conjecture has been shown to hold for $p$ large
(without an explicit bound) thanks to work of Andersen, Jantzen and Soergel
\cite{AJS}, Kashiwara and Tanisaki \cite{KT1,KT2}, Kazhdan and
Lusztig \cite{KLaffine,KL2,KL3} and Lusztig \cite{LuM}. Alternative
proofs for large $p$ have been given by Arkhipov, Bezrukavnikov and 
Ginzburg \cite{ABG}, Bezrukavnikov, Mirkovic and Rumynin 
\cite{BMR, BM} (in the broader context of Lie algebra representations)
and Fiebig \cite{F}. Fiebig also gives an explicit
enormous\footnote{e.g. at least of the order of $p \gg n^{n^2}$ for
  $\SL_{n+1}$.}  bound \cite{F2}, and
establishes the multiplicity one case \cite{F3}. 
For any fixed $G$ and ``reasonable'' $p$ very
little is known: the case of rank 2 groups can be deduced from
Jantzen's sum formula, and intensive computational efforts have
checked the conjecture for small $p$ and certain groups, all of
rank $\le 5$. There is no conjecture as to what happens if
$p$ is smaller than the Coxeter number.

In \cite{Soe} Soergel introduced a subquotient of the category of
rational representations, dubbed the ``subquotient around the
Steinberg weight'', 
as a toy model for the study of Lusztig's conjecture. Whilst
the full version of Lusztig's conjecture is based on the combinatorics of alcoves
and the affine Weyl group, the subquotient around the Steinberg weight
is controlled by the finite Weyl group, and behaves like a modular
version of category $\OC$.  Lusztig's conjecture implies that the
multiplicities in the subquotient around the Steinberg weight are given
by finite Kazhdan-Lusztig polynomials.  Thus Lusztig's conjecture
implies that ``the subquotient around the Steinberg weight satisfies the
Kazhdan-Lusztig conjecture''.

In \cite{Soe} Soergel goes on to explain how the subquotient around the Steinberg
weight is controlled by Soergel bimodules. This allows him to relate
this category to the category of constructible sheaves on the Langlands dual flag
variety, with coefficients in the field of definition of $G$. Using
Soergel's results and the theory of parity sheaves \cite{JMW2}, one
can see that a part of Lusztig's conjecture for $p > h$ is equivalent to absence of $p$ torsion in the stalks and
costalks of integral intersection cohomology complexes of Schubert
varieties in the flag variety. It has been known since the birth of
the theory of intersection cohomology that $2$-torsion occurs in type $B_2$, and $2$- and $3$-torsion occurs in type $G_2$. For over a decade no
other examples of torsion were known. In
2002 Braden discovered 2-torsion in the stalks of integral
intersection cohomology complexes on flag varieties of types $D_4$ and
$A_7$ (see Braden's appendix to \cite{W}). In 2011 Polo discovered 3-torsion in the
cohomology of the flag variety of type $E_6$ and 
$n$-torsion in a flag variety of type $A_{4n-1}$. Polo's
(as yet unpublished) results are significant, as they
emphasise how little we understand in high rank (see the final lines of \cite{WO}).

In general these topological calculations appear extremely
difficult. Recently Elias and
the author found a presentation for the monoidal category of Soergel bimodules
by generators and relations \cite{EW}, building on the work of
Libedinsky \cite{LibRA}, Elias-Khovanov \cite{EKh} and Elias \cite{EDC}.
One of the applications of this theory is
that one can decide whether a given intersection cohomology complex
has $p$-torsion in its stalks or costalks (the bridge between intersection
cohomology and Soergel bimodules is provided by the theory of parity
sheaves).\footnote{One can also perform this calculation using the theory of moment
  graphs \cite{FW}. However the computations using generators and
  relations are generally much simpler.}
 The basic idea is as follows: given any pair $(\un{w},x)$,
where $x, w \in W$ and $\un{w}$ is a reduced expression for $w \in W$, 
one has an ``intersection form'', an integral matrix.  The stalks
of the intersection cohomology complex corresponding to $w$ are free of $p$-torsion
if no elementary divisors of the intersection forms
associated to all elements $x \le w$ are divisible by $p$.
In principle, this gives an algorithm to decide whether
Lusztig's conjecture is correct around the Steinberg
weight.\footnote{One can extend this to the full version of Lusztig's
  conjecture by using a certain subset of the affine Weyl group,
  thanks to the work of Fiebig
  \cite{F}. Although it seems likely that the converse holds, at present one only knows one implication: the
  absence of $p > h$ torsion implies the truth of Lusztig's
  conjecture in characteristic $p$.} 
This algorithm (in a slightly
different form) was discovered independently by Libedinsky
\cite{LLC}.

The generators and relations approach certainly makes calculations
easier. However this
approach still has its difficulties: the diagrammatic calculations remain
extremely subtle, and the ``light leaves'' basis in which the intersection form is
calculated depends on additional choices which seem difficult to make
canonical. Recent progress in this direction has been made by Xuhua He
and the author \cite{HW}, who discovered that certain entries in the
intersection form (which in some important examples are all entries)
are canonical and may be evaluated in terms of expressions
in the nil Hecke ring.

The main result of this paper may be seen as an example of this
phenomenon. We observe that one may embed
certain structure constants of Schubert calculus for $\SL_n$ as the
entries of $1 \times 1$ intersection forms associated to pairs
$(\un{w}, x)$ in (much) higher rank groups. In this way one can produce
many new examples of torsion which grow exponentially in the rank.
For example, using Schubert calculus for the flag variety of 
$\SL_4$ we observe that the Fibonacci number $F_{i+1}$ occurs
as torsion in $\SL_{3i+5}$. We deduce that there is no linear function
$f(n)$ of $n$ such that Lusztig's conjecture holds for all $p \ge
f(n)$ for $\SL_n$. In the appendix (by Kontorovich, McNamara and the author) we
apply  recent results of Bourgain-Kontorovich \cite{BourgainKontorovich2014} in number
theory to deduce that the torsion in $\SL_n$ grows exponentially in $n$.

Finally, there is a related conjecture due to James \cite{James}
concerning the simple representations of the symmetric group in characteristic
$p$. When combined with known results about the decomposition numbers 
for Hecke algebras at roots of unity, the James conjecture would yield the
decomposition numbers for symmetric groups $S_n$ in characteristic $p
> \sqrt{n}$. In the final section of the paper we explain why our counter-examples to Lusztig's conjecture for $SL_N$ with
$p > {N \choose 2}$ imply that the James conjecture is
incorrect for $S_{p{N \choose 2}}$. (Parts of this section were
explained to me by Joe Chuang.)

\subsection{Main result} \label{sec:main}

Let $R = \ZM[x_1, x_2, \dots, x_n]$ be a polynomial ring in $n$
variables. We
regard $R$ as a graded ring with $\deg x_i =2$ (we double degrees for
reasons coming from Soergel bimodules). Let
$W = S_n$ denote the symmetric group on $n$ letters viewed as a
Coxeter group with simple reflections $S$ consisting of the simple transpositions.
Then $W$ acts by
permutation of variables on $R$. Let $s_1,
\dots, s_{n-1}$ denote the simple transpositions of $S_n$ and let
$\ell$ denote the corresponding length function. Let $\partial_i$ denote the $i^{th}$ divided
difference operator:
\[
\partial_i(f) = \frac{f - s_i f}{x_i-x_{i+1}} \in R.
\]
For any element $w \in S_n$ we obtain well-defined
  operators $\partial_w = \partial_{i_1}
\dots \partial_{i_m}$ where $w = s_{i_1} \dots s_{i_m}$ is a reduced
expression for $w$ in the generators $S$.

Consider elements of the form
\[
\kappa = \partial_{w_m} (x_1^{a_m}x_n^{b_m}  \partial_{w_{m-1}}
(x_1^{a_{m-1}}x_n^{b_{m-1}} 
\dots \partial_{w_1} (x_1^{a_1}x_n^{b_1}) \dots ))
\]
where $w_i \in S_n$ are arbitrary. 
We assume that $\sum \ell(w_i) = a + b$ where $a = \sum
a_i$ and $ b = \sum b_i$ so that $\kappa \in \ZM$ for degree reasons.
Given a subset $I \subset \{ 1, \dots, n-1\}$ let $w_I$ denote the
longest element in the parabolic subgroup $\langle s_j \rangle _{j \in I}$. Our main theorem is the following:

\begin{thm} \label{thm:main}
Suppose that $n \ge 1$ and $a, b \ge 0$ are as above, and that $\kappa \ne 0$.
Then there exists a reduced expression $\underline{w}$ for an element
of $S_{a+n+b}$ such that the intersection form in degree zero of $\underline{w}$ at
$w_I$, where $I = \{ 1,2, \dots, a + n + b-1\} \setminus \{ a, a + n
\}$, is the
$1 \times 1$ matrix $((-1)^a \kappa)$.
\end{thm}

The construction of the expression $\un{w}$ is explicit and
combinatorial based on $w_1, \dots, w_m$, $a_1, \dots, a_m$ and
$b_1, \dots, b_m$. We will also see that for certain
  choices of $a_i, b_i, w_i$ the prime factors of the numbers $\kappa$
grow exponentially in $h = n + a + b$.

\subsection{Schubert calculus} \label{sec:Schubert}
We explain why ``Schubert calculus'' occurs in
the title. Consider the coinvariant ring $C$ for the action of $W = S_n$
on $R$. That is, $C$ is equal to $R$ modulo the ideal generated by
$W$-invariant polynomials of positive degree. The Borel isomorphism
gives a canonical identification of $C$ with the integral cohomology
of the complex flag variety of $\SL_n$.

The divided difference operators $\partial_w$ act on $C$,
as do elements of $R$. 
The coinvariant ring $C$ has a graded $\ZM$-basis given by the
Schubert classes $\{ X_w \;| \;w \in S_n \}$ (normalised with $X_{w_0} = x_1^{n-1}x_2^{n-2}
\dots x_{n-1}$ and $X_w = \partial_{ww_0} X_{w_0}$). 
We have:
\begin{equation} \label{eq:dem}
\partial_i X_w = \begin{cases} X_{s_iw} & \text{if $s_iw < w$,} \\ 0 &
  \text{otherwise.} \end{cases}
\end{equation}
The action of multiplication by $f \in R$ of degree two is given as
follows (the Chevalley formula):
\begin{equation} \label{eq:Chevalley}
f \cdot X_w = \sum_{t \in T \atop \ell(tw) = \ell(w) + 1 } \langle f,
  \alpha_t^\vee \rangle X_{tw}.
\end{equation}
(Here $T$ denotes the set of reflections (transpositions) in $S_n$ and
if $t = (i,j) \in T$ with $i < j$ then
$\alpha_t^\vee = \e_i - \e_j$ where $\{\e_i\}$ is the dual basis to $x_1,
\dots, x_n$.)

Now consider the numbers one may obtain as coefficients in the basis
of Schubert classes by multiplication by $x_1$ and
$x_n$ and by applying Demazure operators, starting with $X_{\id}$.  
Because $\partial_{w}X_{w^{-1}} = X_{\id} = 1$, any coefficient of any
Schubert class that we obtain in this way can be realised as the coefficient of $X_{\id}$.
Now
Theorem \ref{thm:main} says that any such number occurs as torsion in
$\SL_{n+a+b}$ where $a$ (resp. $b$) counts the number of times that one
has applied the operator of multiplication by $x_1$ (resp. $x_n$).



\subsection{Note to the reader} This paper is entirely algebraic in
that it relies only on Soergel (bi)modules, their diagrammatics and connections to
representation theory (due to Soergel). Except in remarks, we neither explain nor use the
relation to constructible sheaves and torsion. An alternative
geometric proof of the main theorem (discovered a year after this
paper was first circulated) is given in \cite{WT}.

\subsection{Structure of the paper}

\begin{description}
\item[\S\ref{sec:SB}-\ref{sec:IF}] Contains background on Soergel
  (bi)modules and intersection forms.
\item[\S\ref{sec:proof}] We prove the main theorem.
\item[\S\ref{sec:ex}] We use our main theorem for $n = 4, 5$ to give
  examples of torsion.
\item[\S\ref{sec:lusztig}] We explain the connection to the Lusztig
  conjecture.
\item[\S\ref{sec:james}] We explain the connection to the James
  conjecture.
\item[\S\ref{app}] We (AK, PM and GW) prove exponential growth of
  torsion. 
\end{description}

\subsection{Acknowledgements} The ideas of this paper crystallised after
long discussions with Xuhua He. I would like to thank him warmly for
the invitation to Hong Kong and the many interesting discussions that
resulted from the visit. I
would also like to thank Ben Elias for countless hours (often
productive, always enjoyable) getting to know Soergel
bimodules. His influence is omnipresent in this paper.

Thanks to Joe Chuang for useful correspondence and explaining how to get counter-examples in the
symmetric group. Finally, thanks to Henning Haahr Andersen, Ben Elias, Peter Fiebig, Anthony Henderson, Daniel
Juteau, Nicolas Libedinsky, Kaneda Masaharu and especially
Patrick Polo and the referees for valuable comments.
These results were announced in June 2013 at ICRT VI in Zhangjiajie,
China.

\section{Soergel bimodules} \label{sec:SB}

In the first three sections we recall what we need from the theory of Soergel (bi)modules and
intersection forms. This paper is not self-contained. The main
references are \cite{S90, SHC, SB, EKh, EDC, EW}.

Fix $n \ge 1$ and let $W = S_n$ denote the symmetric group on $n$
letters. Throughout we view $W$ as a Coxeter group with simple
reflections $S = \{ (i,i+1) \; | \; 1 \le i < n \}$, and denote by
$\ell$ its length function and $\le$ its Bruhat order. Let $\HC$ denote the
Hecke algebra of $(W,S)$ over $\ZM[v^{\pm 1}]$ normalised as in 
\cite{SoeKL}. Let $\{ H_x \}_{x \in W}$ and $\{\un{H}_x\}_{x \in W}$ denote its
standard and Kazhdan-Lusztig bases.

Fix a field $\Bbbk$ of characteristic $p > 2$ and let $R =
\Bbbk[x_1, \dots, x_n]$. Then $W$ acts by permutation of variables on
$R$ (graded algebra automorphisms). The reader may
  easily check (see e.g. \cite[Lemma
  7.4]{LLC}) that this action is reflection faithful in the sense
  of \cite[Definition 1.5]{SB}. Given $s
\in S$ we denote by $R^s \subset R$ the invariant subring.

Given a $\ZM$-graded object (vector space, module, bimodule) $M = \bigoplus
M^i$ we let $M(j)$ denote the shifted object: $M(j)^i = M^{i+j}$.  

The category of \emph{Soergel bimodules} $\BC$ is the full additive monoidal graded
Karoubian subcategory of graded $R$-bimodules generated by $B_s = R \otimes_{R^s} R(1)$ for
all $s \in S$. 
In other words, the indecomposable Soergel bimodules are the shifts of the
indecomposable direct summands of the \emph{Bott-Samelson bimodules}
\[
B_{\un{w}} = B_{s_1} \otimes_R B_{s_2} \otimes_R \dots \otimes_R B_{s_m}(m)
\]
for all expressions $\un{w} = s_1s_2 \dots s_m$ in $S$. For any $w \in S_n$
let $B_w$ denote the indecomposable self-dual Soergel bimodule which
occurs as a summand of $B_{\un{w}}$ for any reduced expression
$\un{w}$ for $w$, and is not isomorphic to a summand of
$B_{\un{w}'}$ for any shorter $\un{w}'$. The set $\{ B_w \}_{w \in W}$  coincides with
the set of all indecomposable Soergel bimodules, up to shifts in the
grading \cite{SB}.

\begin{remark}
  In \cite{SB} Soergel develops the theory of Soergel bimodules for a
  reflection faithful representation $V$ over an infinite field of
  characteristic $\ne 2$. We
  have remarked above that the reflection faithful hypothesis is
  always satisfied. The
  assumption that $\Bbbk$ is infinite is made
  in order to identify $R$ with the polynomial functions on $V$. However all the results of \cite{SB} hold if one
  simply defines $R$ to be the symmetric algebra on $V^*$, as we do. Alternatively, the reader may assume that
  $\Bbbk$ is infinite throughout.
\end{remark}

We denote by $[\BC]$ the split Grothendieck group of $\BC$ (i.e. $[B]
= [B'] + [B'']$ if $B \cong B' \oplus B''$). We make $[\BC]$ into a
$\ZM[v^{\pm 1}]$ algebra via $v[M] := [M(-1)]$, $[B][B'] := [B
\otimes_R B']$. In \cite{SB} Soergel
proves that there exists a unique isomorphism of $\ZM[v^{\pm
  1}]$-algebras
\[
\ch : [\BC] \simto \HC
\]
such that  $\ch(R(-1)) = v$ and $\ch(B_s) = \un{H}_s$ for all $s \in S$.

\begin{remark}
  Under our assumptions $B_w$ may be
  realised as the equivariant intersection cohomology of the
  indecomposable parity sheaf \cite{JMW2} of the Schubert variety labelled
  by $w$ in the flag variety \cite{Fcom, FW}. In particular, if $\Bbbk$ is of
  characteristic zero, then $B_w$ is the equivariant intersection
  cohomology of a Schubert variety. In fact the whole theory of
  Soergel bimodules can be seen as providing an algebraic description
  of the Hecke category.
\end{remark}

Set $\p\un{H}_x := \ch(B_x) \in \HC.$ Then $\{ \p\un{H}_x \}_{x \in W}$ is a
basis which only depends on the characteristic $p$ of $\Bbbk$, the
\emph{$p$-canonical basis} (see \cite{WO,JW}). Let us write
\begin{gather} \label{eq:pcan}
\un{H}_x = \sum h_{y,x} H_y, \qquad
\p\un{H}_x = \sum \p h_{y,x} H_y, \qquad
\p\un{H}_x = \sum {}^pa_{y,x} \un{H}_y.
\end{gather}
for polynomials $h_{y,x} \in \ZM[v]$ and $\p h_{y,x}, {}^pa_{y,x} \in
\ZM[v^{\pm 1}]$.   The polynomials $h_{y,x}$ are (normalisations of)
  Kazhdan-Lusztig polynomials and have non-negative
  coefficients. The polynomials $\p h_{y,x}, {}^pa_{y,x}$ also have
  non-negative coefficients \cite[Proposition 4.2]{JW}.

Throughout this paper we will say that $p$
\emph{occurs as torsion in $\SL_n$} if there exists $x \in W$ such
that $\p\un{H}_x \ne \un{H}_x$.



\section{Soergel modules} \label{sec:SM}
In this section we assume that $p > n$, so that the results of
\cite{Soe} are available.

Let $R^W_+ \subset R$ denote the $W$-invariants
of positive degree, $\langle R^W_+\rangle$ the ideal they generate, and
$C = R/\langle R^W_+\rangle$ the coinvariant
algebra, which inherits an (even) grading and a $W$-action from $R$. 
Let $\CC$ denote the category of \emph{Soergel
  modules} consisting of all
\[
D_{\un{w}} := C \otimes_{C^{s_m}} \dots \otimes_{C^{s_2}} C
\otimes_{C^{s_1}} \Bbbk(m)
\]
for expressions $\un{w} = s_1s_2 \dots s_m$ in $S$, together with
their shifts, direct sums and summands inside the category of graded
$C$-modules. (Note the order of
tensor factors.)

For a reduced expression $\un{x}$ for $x$ let $D_x$ denote the
unique summand of $D_{\un{x}}$ which does not occur as a summand of
$D_{\un{x}'}$ for any shorter expression $\un{x}'$. The set $\{ D_x \;
| \; x \in W \}$ is well-defined and gives a set of representatives
for the isomorphism classes of indecomposable Soergel modules (up to
shift) \cite[Theorem 2.8.1]{Soe}.

How to go from Soergel bimodules to Soergel modules? Given a right $R$-module $M$ which is killed by
$R^W_+$ the canonical map $M \otimes_{R^s} R \onto M
\otimes_{C^s} C$ is an isomorphism. Hence we have an
isomorphism of graded right $C$-modules:
\[
\Bbbk \otimes_R R \otimes_{R^s} R \otimes_{R^t} \dots \otimes_{R^u} R
\cong \Bbbk \otimes_C C \otimes_{C^s} C
\otimes_{C^t} \dots \otimes_{C^u} C.
\]
It follows that if we compose the functor $M \mapsto \Bbbk \otimes_R M$ with the
equivalence between right and left $C$-modules ($C$ is commutative) we
obtain a functor
\[
c : \BC \to \CC
\]
with $c(B_{\un{w}}) = D_{\un{w}}$.

\begin{lem}
  $c(B_x) \cong D_x$.
\end{lem}

\begin{proof} \emph{Step 1:} We claim that the natural map provides an isomorphism:
\begin{gather} \label{eq:bsiso}
\Bbbk \otimes_R \Hom^\bullet_{\BC}(B_{\un{x}}, B_{\un{y}}) \simto \Hom^\bullet_{C}(
c(B_{\un{x}}), c(B_{\un{y}})).
\end{gather}
(Here and in the rest of the proof, $\Hom^\bullet$ denotes the graded
module of morphisms of all degrees.)
By repeated application of the biadjunction ($\otimes_R B_s(1), \otimes_R
B_s(1)$) we may assume that $\un{x}$ is the empty sequence.
The map $\phi \mapsto \phi(1)$ gives a canonical identification of
$\Hom^\bullet_\BC(R, B_{\un{y}})$ with the submodule of invariants
\[
\Gamma_{\id} B_{\un{y}} := \{ m \in B_{\un{y}} \; | \; rm = mr \text{ for all $r \in R$}\}.
\]
Now $\Gamma_{\id} B_{\un{y}}$ is the first step in the filtration
$\Gamma_{\le 0} B_{\un{y}} \subset \Gamma_{\le 1} B_{\un{y}} \subset
\dots$ considered after the proof of \cite[Proposition 5.7]{SB}, and from
\cite[Proposition 5.9]{SB} we deduce that the subquotients of this
filtration are free as left $R$-modules. Thus $\Gamma_{\id} B_{\un{y}}$ is a summand of $B_{\un{y}}$ as a left
$R$-module. The injectivity of $\eqref{eq:bsiso}$ follows.


We deduce the surjectivity of \eqref{eq:bsiso} by showing that both sides have the same
(finite) dimension.
If $\un{y} = s_1s_2 \dots s_m$ let us write
$\un{H}_{s_1} \un{H}_{s_2} \dots \un{H}_{s_m} =
\sum g_x H_x$
for some $g_x \in \ZM[v^{\pm 1}]$. In the notation of \cite{SB}
we have, by \cite[Proposition 5.7]{SB},
\[
\sum_{m \in \ZM} (B_{\un{y}} : \nabla_{\id}[m])v^{-m} = g_{\id}
\]
and $(R:\Delta_x[m]) =
\delta_{x,{\id}}\delta_{m,0}$ (Kronecker's $\delta$). Now we apply
\cite[Theorem
5.15]{SB} to deduce that $\Hom^\bullet(R, B_{\un{y}})$ is free of rank
$g_{\id}(1)$ over $R$. On the other hand, if $\langle -, -\rangle : \ZM
W \times \ZM W \to \ZM$ denotes the pairing with $\langle x, y
\rangle = \delta_{x,y}$, then by \cite[Lemma 2.11.2]{Soe}, we have
\[
\dim \Hom^\bullet_C( \Bbbk, c(B_{\un{y}})) = \langle {\id}, \sum g_x(1) x \rangle = g_{\id}(1).
\]
The surjectivity follows.

\emph{Step 2:} Because $c(B_{\un{w}}) = D_{\un{w}}$ we can appeal to
the defining properties of $B_x$ and $D_x$ to see that it is enough to show: if $B$ is
indecomposable, then so is $c(B)$. By the previous step
$\End_C(c(B)) = \Bbbk \otimes_R \End_\BC(B)$, and so $\End_C(c(B))$ is a quotient of 
the local ring $\End_\BC(B)$. Now the result
follows as a non-zero quotient of a local ring is local.
\end{proof}

\begin{remark}
  The above proof uses representation theory, via
  \cite[Lemma 2.11.2]{Soe}. Soergel has found an algebraic proof of
  the above lemma, valid for any finite Coxeter group
  (unpublished).
\end{remark}

Denote by $f$ the functor of forgetting the grading on a $C$-module,
and let $f\CC$ denote the essential image of $\CC$ under $f$. By
\cite[Theorem 2.8.1]{Soe}, 
the indecomposable objects in $f \CC$ are precisely the $\{ fD_x \}$.
We denote by
$[f\CC]$, $[\CC]$ the split Grothendieck groups of $f\CC$ and
$\CC$ respectively. Because $\CC$ is graded, $[\CC]$ is naturally a
$\ZM[v^{\pm 1}]$-module via $v[M] := [M(-1)]$ as above.
These observations, together with the above lemma, show that we have a commutative diagram:
\begin{gather*}
  \begin{tikzpicture}[xscale=2.5,yscale=1.7]
    \node (fc) at (1,1) {$[f\CC]$};
    \node (c) at (2,1) {$[\CC]$};
    \node (d) at (3,1) {$[\BC]$};
\node (b2) at (1,0) {$\ZM S_n$};
\node (b3) at (2,0) {$\HC$};
\node at (2.5,0) {$=$};
\node (b4) at (3,0) {$\HC$};
\draw[->] (c) to node[above] {$f$} (fc);
\draw[->] (d) to node[above] {$c$} node[below]{$\sim$} (c);
\draw[->] (fc) to node[left] {$\b$} node[right] {$\sim$}  (b2);
\draw[->] (c) to node[right] {$\sim$}  (b3);
\draw[->] (d) to node[right] {$\ch$} node[left] {$\sim$} (b4);
\draw[->] (b3) to node[below] {$1 \mapsfrom v$} (b2);
  \end{tikzpicture}
\end{gather*}
where $\b$ is defined by
\begin{equation} \label{eq:beta}
\b(fD_x) = \ch(B_x)_{|v = 1} = {\p\un{H}_x}_{|v = 1}.
\end{equation}

\section{Intersection forms} \label{sec:IF}

Let $\BC$
denote the category of Soergel bimodules defined above. Given any
ideal $I \subset W$ (i.e. $x \le y \in I \Rightarrow x \in
I$)  we denote by $\BC_I$ the ideal of $\BC$
generated by all morphisms which factor through a Bott-Samelson
bimodule $B_{\un{y}}$, where $\un{y}$ is a reduced expression for $y
\in I$. 

Given $x \in W$ we denote by $\BC^{\ge x}$ the quotient category
$\BC / \BC_{\not \ge x}$ where $\not \ge \hspace{-0.1cm} x := \{ y | y \not \ge
x \}$. We write $\Hom_{\ge x}$ for (degree zero) morphisms in
$\BC^{\ge x}$. All Bott-Samelson bimodules $B_{\un{x}}$ corresponding to 
reduced expressions $\un{x}$ for $x$ become canonically isomorphic to
$B_x$ in $\BC^{\ge x}$. We have $\End_{\ge x}(B_x) = \Bbbk$.
Given any expression $\un{w}$ in $S$ the \emph{intersection form}\footnote{The
terminology ``intersection form'' comes from geometry: in de Cataldo and
Migliorini's Hodge theoretic proof of the decomposition theorem, a key
role is played by certain intersection forms
associated to the fibres of proper maps \cite{dCM, dCM2}. In our setting, these
intersection forms are associated to the fibres of Bott-Samelson
resolutions of Schubert varieties. The relevance of these forms for
the study of torsion in intersection cohomology was pointed out in
\cite{JMW2}.}
is the canonical pairing
\[
I_{x,\un{w},d}^{\Bbbk} : \Hom_{\ge x}(B_x(d), B_{\un{w}}) \times \Hom_{\ge x}(B_{\un{w}}, B_x(d))
\to \End_{\ge x}(B_x(d)) = \Bbbk.
\]

The following is standard (see e.g. \cite[Lemma 3.1]{JMW2} for a
similar situation):

\begin{lem} \label{lem:rank}
  The multiplicity of $B_x(d)$ as a summand of $B_{\un{w}}$ equals
  the rank of $I^{\Bbbk}_{x,\un{w},d}$.
\end{lem}


In the papers \cite{EKh, EW} the category of Soergel
bimodules is presented by generators and relations. More precisely,
a diagrammatic category is defined by generators and relations and it
is proved that its Karoubi envelope is equivalent to Soergel bimodules,
as a graded monoidal category. We will not repeat the rather
complicated list of generators and relations here, see \cite [\S
1.4]{EW} or \cite[\S 2.7]{HW}.

In the category $\DC$ the intersection form is explicit and amenable
to computation: see \cite[\S 2.10]{HW} for examples. From the
diagrammatic approach it is clear that $I_{x,\un{w},d}^{\Bbbk}$ is
defined over $\ZM$, in the sense that there exists an integral form
$I_{x,\un{w},d}$ on a pair of free $\ZM$-modules such that
$I_{x,\un{w},d}^{\Bbbk} = I_{x,\un{w},d} \otimes_{\ZM} \Bbbk$ for any
field $\Bbbk$.




\begin{cor} The following are equivalent: \label{prop:ch}
  \begin{enumerate}
  \item The indecomposable Soergel bimodules in characteristic $p$ categorify the
    Kazhdan-Lusztig basis. That is, $\p\un{H}_x = \un{H}_x$ for all $x \in W$.
\item For all (reduced) expressions $\un{w}$, all $x \in W$ and all $m
  \in \ZM$ the graded ranks of $I_{x,\un{w},m} \otimes_{\ZM} \QM$ and
  of $I_{x,\un{w},m} \otimes_{\ZM} \Bbbk$ agree.
\item For all reduced expressions $\un{w}$ and all
  $x \in W$ the graded ranks of $I_{x,\un{w},0} \otimes_{\ZM} \QM$ and 
  of $I_{x,\un{w},0} \otimes_{\ZM} \Bbbk$ agree.
  \end{enumerate}
\end{cor}

\begin{proof}
Soergel's theorem \cite[Lemma 5]{SL} implies that the
indecomposable Soergel bimodules categorify the 
Kazhdan-Lusztig basis if $\Bbbk$ is of characteristic zero (see
  \cite{EW2} for an algebraic proof of this fact). 
Now Lemma \ref{lem:rank} says that (2) holds if and only if $B_{\un{w}}$ decomposes the same
way over $\QM$ as it does over $\Bbbk$. Hence (1) and (2) are
equivalent and (1) implies (3).

It remains to see that (3) implies (1). We show the contrapositive. So
assume that (1) is not satisfied, and let $w$ be of minimal
length such that $\p\un{H}_w \ne \un{H}_w$. For any $s \in S$ with $ws
< w$, $B_w$ is a summand of $B_{ws}B_s$. By our minimality assumption
$\ch(B_{ws}) = \un{H}_{ws}$ and hence
$\ch(B_{ws}B_s) = \un{H}_{ws}\un{H}_s = \sum g_x \un{H}_x$ for some $g_x \in
\ZM_{\ge 0}$. Hence $\ch(B_w) = \p\un{H}_w =\sum \p a_{x,w} \un{H}_x$ with $\p a_{x,w} \in
\ZM_{\ge 0}$. By Lemma \ref{lem:rank}, if $x < w$ is such
that $\p a_{x,w} \ne 0$ then the ranks of $I_{x,\un{w},0}\otimes_{\ZM}
\QM$ and $I_{x,\un{w},0}\otimes_{\ZM} \Bbbk$ differ, for any
reduced expression $\un{w}$ for $w$. Thus (3) $\Rightarrow$ (1).
\end{proof}


\begin{remark} The intersection form and the above
  proposition is one of the tools used by Fiebig to establish
  his bound \cite{F2}. 
\end{remark}

\section{Proof of the theorem} \label{sec:proof}

Let $W$ denote the symmetric group on $\{ 1, 2, \dots, a + n + b\}$
with Coxeter generators $S = \{ s_1, s_2, \dots,
s_{a+n+b-1}\}$ the simple transpositions. Given a subset $I \subset S$
let $W_I$ denote the corresponding standard parabolic subgroup and
$w_I$ its longest element. 
Consider the sets
\[
A = \{s_1,s_2, \dots, s_{a-1}\}, M = \{ s_{a+1}, \dots, s_{a+n-1}\}, B = \{ s_{a+n+1},
\dots, s_{a+n+b-1}\}.
\]
Then $W_A$ (resp. $W_M$, resp. $W_B$) is the subgroup of permutations
of $\{ 1, \dots, a\}$ (resp. $\{ a+1, \dots, a+n \}$,
resp. $\{a+n+1, \dots, a + n + b \}$).

We use the notation of \S \ref{sec:main} except we shift all indices
by $a$. That is we regard $S_n$ as embedded in $S_{a+n+b}$ as the
standard parabolic subgroup $W_M$. We rename $R =
\ZM[x_1, \dots, x_{a+n+b} ]$ and write $\alpha_i = x_i - x_{i+1}$ for
the simple root corresponding to $s_i$. Fix 
\begin{equation}
  \label{eq:kappa}
\kappa = \dd_{w_m} (x_{a+1}^{a_m}x_{a+n}^{b_m}  \dd_{w_{m-1}}
(x_{a+1}^{a_{m-1}}x_{a+n}^{b_{m-1}} 
\dots \dd_{w_1} (x_{a+1}^{a_1}x_{a+n}^{b_1}) \dots ))  
\end{equation}
which we assume is a non-zero integer. (Now $w_1, \dots, w_m \in
W_M$ and the fact that $\kappa$ is a non-zero integer implies
that $\sum \ell(w_i) = a + b$.)

We now perform some preliminary simplifications of
the right hand side of \eqref{eq:kappa}. 
By replacing
each $x_{a+1}^{a_i}x_{a+n}^{b_i}$ with 
  $x_{a+1}^{a_i}\partial_{\id{}}x_{a+n}^{b_i}$ we
  may  assume that for all
  $i$, either $a_i$ or $b_i$ is zero. Let $M' = M \setminus \{ s_{a+1},
  s_{a+n-1}\}$. If $w \in W_{M' \cup \{ s_{a+1} \}}$ then $\partial_w$
    commutes with the operator of multiplication with $x_{a+n}$. Thus
    if $a_i$ is zero then we may assume that $w_i$ is minimal in its
    coset $w_i W_{M' \cup \{ s_{a+1} \}}$. Similarly, if $b_i$ is zero
      then we may assume that $w_i$ is minimal in its coset $w_i W_{M'
        \cup \{ s_{a + n -1} \}}$. From now on we assume that
    the right hand side of \eqref{eq:kappa} has been simplified in
    this way. Finally, the minimal coset representatives of
  $W_M / W_{M' \cup \{ s_{a+n-1}\}}$ are the elements:
\[
\id,
  s_{a+1}, \; s_{a+2}s_{a+1}, \; \dots, \; s_{a+n-1}s_{a+n-2} \dots s_{a+2}s_{a+1}.
\]
Similarly, the minimal coset representatives of $W_M / W_{M' \cup \{ s_{a+1}\}}$ are the elements:
\[
\id,
  s_{a+n-1}, \; s_{a+n-2}s_{a+n-1}, \; \dots, \; s_{a+1}s_{a+2} \dots s_{a+n-2}s_{a+n-1}.
\]
Thus each $w_i$ belongs to the first (resp. second) list if $b_i = 0$
(resp. $a_i =0$).

Fix a reduced expression $\un{w}_M$ for $w_M$ and reduced expressions
$\un{w}_i$ for each $w_i$. (In fact, following the reductions of
  the previous paragraph each $w_i$ has a unique reduced
  expression.)
Let $\un{w}$ be the sequence
\[
\un{w} =  \un{w}_m\un{u}_m \un{v}_m\dots \un{w}_2\un{u}_2 \un{v}_2\un{w}_1\un{u}_1\un{v}_1\un{w}_M
\]
where 
\begin{gather*}
  \un{u}_1 = ( s_a\dots s_{a-a_1+1} ) \dots (s_a s_{a-1}) (s_a) \\
\un{u}_2 = (s_a\dots s_{a-a_1-a_2+1})  \dots
(s_a \dots s_{a-a_1-1})(s_a \dots s_{a-a_1})
\\
\vdots \\
\un{u}_m = (s_a  \dots s_1) \dots (s_a \dots s_{a-a_1-\dots-a_{m-1} - 1}) 
(s_a  \dots s_{a-a_1 - \dots - a_{m-1}} ) 
\end{gather*}
(subscripts fall by 1 within each parenthesis, and $s_a$ occurs $a_i$ times in $\un{u}_i$) and
\begin{gather*}
  \un{v}_1 = (s_{a+n} \dots s_{a+n+b_1-1}
 )  \dots (s_{a+n} s_{a+n+1}) (s_{a+n})\\
\un{v}_2 = (s_{a + n} \dots s_{a + n + b_1 + b_2 - 1} ) \dots (
s_{a+n} \dots s_{a+n+b_1+1}) (s_{a + n} \dots s_{a+n+b_1} )
 \\
\vdots \\
\un{v}_m = 
(s_{a+n}  \dots s_{a + n + b - 1}) \dots
(s_{a + n} \dots s_{a + n + b_1 + \dots + b_{m-1}}) 
\end{gather*}
(subscripts rise by 1 within each parenthesis, and $s_{a+n}$ occurs $b_i$
times in $\un{v}_i$).

\begin{remark}
  The sequence $\un{u}_m \dots \un{u}_2\un{u}_1$ (resp. $\un{v}_m
  \dots \un{v}_2\un{v}_1$) is a reduced expression for the longest
  element of $W_{A \cup \{ s_a \}}$ (resp. $W_{\{s_{a+n} \} \cup
    B}$). If we denote by $\un{u}'_i$ (resp. $\un{v}'_i$) the expression obtained
  from $\un{u}_i$ (resp. $\un{v}_i$) by deleting every occurrence of
  $s_a$ (resp. $s_{a+n}$) then $\un{u}'_m \dots \un{u}'_2\un{u}'_1$ (resp. $\un{v}'_m
  \dots \un{v}'_2\un{v}'_1$) is a reduced expression for the longest
  element of $W_A$ (resp. $W_B$).
\end{remark}

\begin{ex}
  We give a real-life example. We take $n = 4$ and consider the operator $F: f
  \mapsto 
  \dd_{23}(x_4^2(\dd_1(x_1 f)))$ on $\ZM[x_1, x_2,x_3, x_4]$ (we write
  $\dd_{23} := \dd_{2}\dd_3$). In the next section we will see that
  $F$ is a ``Fibonacci operator''; in particular
\[
\dd_1F^3(x_1) =
\dd_{123}(x_4^2\dd_1(x_1\dd_{23}(x_4^2\dd_1(x_1\dd_{23}(x_4^2\dd_1(x_1^2))))))=
3.
\]
In the notation of \S \ref{sec:main} we have $w_1 = w_3 = w_5 = s_1$,
$w_2 = w_4 = s_2s_3$, $w_6 = s_1s_2s_3$, $a_1 = 2$, $a_3 = a_5 = 1$,
$a_2 = a_4 = a_6 = 0$, $b_1 = b_3 = b_5 = 0$ and $b_2 = b_4 = b_6 =
2$. Hence $a = 4, b = 6$ and $ a + n + b = 14$. We can
depict $\un{w}$ as follows:
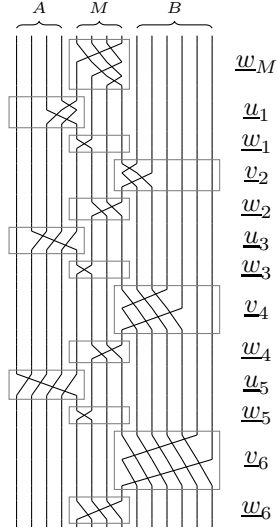
\[
\input{perm.tex}
\]
\end{ex}


The rest of this section will be occupied with the proof of the
following theorem:

\begin{thm} \label{thm:main2}
The degree zero intersection form of $\un{w}$ at $w_{A \cup M \cup B}$ is the $1
\times 1$-matrix $((-1)^a
  \kappa)$.\end{thm}

In the proof we will need the notion of subexpression and
defect together with the main result
of \cite{HW}.
 Fix a word $\un{y} = s_{i_1} \dots s_{i_m}$ in $S$. A
\emph{subexpression} of $\un{y}$ is a sequence $\un{e} = e_1 \dots e_m$ with
$e_i \in \{ 0, 1 \}$ for all $i$. We set $\un{y}^{\un{e}} :=
s_{i_1}^{e_1} \dots s_{i_m}^{e_m} \in W$. Any subexpression $\un{e}$
determines a sequence $y_0, y_1, \dots, y_m \in W$ via $y_0 := \id{}$, $y_j
:= s_{i_{m+1-j}}^{e_{m+1-j}}y_{j-1}$ for $1 \le i \le m$ (so $y_m =
\un{y}^{\un{e}}$). Given a subexpression $\un{e}$ we associate a
sequence $d_j \in \{ U, D \}$ (for \emph{U}p, \emph{D}own) via
\[
d_j := \begin{cases} U & \text{if $s_{i_j} y_{m-j}> y_{m-j}$,}\\
D & \text{if $s_{i_j}y_{m-j} < y_{m-j}$.} \end{cases}
\]
Usually we view $\un{e}$ as the decorated sequence $( d_1e_1, \dots,
d_me_m)$. The
\emph{defect} of $\un{e}$ is
\[
\df(\un{e}) := | \{ i \; | \; d_ie_i = U0 \}| - | \{ i \; | \; d_ie_i = D0 \}|.
\]

\begin{remark} \label{rem:opp}
  See \cite[\S 2.4]{EW} for examples and  motivation. We warn the reader that in this paper we work
  from right to left to define the defect, rather than from left to right as in
  \cite[\S 2.4]{EW} and \cite[\S 2.3]{HW}. This change of conventions is necessary to have
  the operators $\partial_i$ act on polynomials on the left. One may
  easily pass between the two possible choices via the symmetry on
  Soergel bimodules which interchanges left and right actions. (In the
diagrammatic language of \cite{EW} this corresponds to flipping
diagrams about the $y$-axis.)
\end{remark}

Recall that the nil Hecke ring $NH$ is defined to be the algebra generated
by $R$ and symbols $\delta_i$ for each $s_i \in S$, and subject
to the relation $\delta_i^2 = 0$ for all $s_i \in S$, the braid
relations and the nil Hecke relation
\[
\delta_i f = s_i(f) \delta_i + \partial_i (f) \quad \text{for all $s_i \in S$ and
$f \in R$.}
\]
As left $R$-modules $NH$ is free with basis $\{\delta_w \}_{w \in W}$,
where $\delta_w := \delta_{i_1} \dots
\delta_{i_k}$ for any reduced expression $w = s_{i_1} \dots s_{i_k}$. The
grading on $R$ extends to a grading on $NH$ with $\deg \delta_w = -2\ell(w)$
for all $w \in W$.

Equipped with this notation we can now give the proof.

\begin{lem} \label{lem:red}
  $\un{w}$ is reduced.
\end{lem}

\begin{proof}
Let us fix an element $x \in W_Q$ where $Q = \{ s_p, s_{p+1}, \dots,
s_{q-1}, s_q \}$ (for some $p, q$ with $1 < p \le q < a + n + b -1$) and a 
reduced expression $\un{x}$ for $x$. Then for any $j$ the expressions
\[
s_j s_{j-1} \dots s_{p-1} \un{x} \quad \text{and} \quad
s_{j} s_{j+1} \dots s_{q+1} \un{x}
\]
are reduced. (For example, one can write a formula
for how the displayed elements act on $1,2, \dots, a + n + b$ in terms
of $x$, and verify that their lengths differ from $\ell(x)$ by $j-p+2$ (resp. $q -
j +2$) by counting inversions. It follows that the lengths of the
displayed expressions agree with the lengths of the underlying
elements, and thus they are reduced.)

From the definition of $\un{w}$ it follows that there exists a
sequence of expressions $\emptyset = \un{x}_0, \un{x}_1, \dots,
\un{x}_r = \un{w}$ such
that each $\un{x}_i$ is
obtained from $\un{x}_{i-1}$ by the procedure of the previous
paragraph. Thus $\un{w}$ is reduced as claimed. 
\end{proof}

 Write $\un{w} = s_{i_1} \dots s_{i_l}$.

\begin{lem} Any subexpression
  ${\un e}$ of $\un{w}$ with $\un{w}^{\un e} = w_{A \cup M \cup B}$ has $e_j =
  0$ if $s_{i_j} \in \{ s_a, s_{a+n} \}$ and $e_j = 1$ if $s_{i_j} \in
  A \cup B$.
\end{lem}

\begin{proof} Let ${\un e}$ denote a subexpression of $\un{w}$ with
  $\un{w}^{\un e} = w_{A \cup M \cup B}$.

Any expression $\un{y}$ for $w_A$ contains a subsequence of
the form $s_{a-1} s_{a-2} \dots s_1$ (think about what happens to $1
\in \{1, \dots, a + n + b\}$). In $\un{w}$, $s_1$ only occurs once. Left of
$s_1$ there is only one occurrence of $s_{a-1}$, $s_{a-2}$, etc. We
conclude that the restriction of ${\un e}$ to $(s_{a}s_{a-1} \dots s_2
s_1)$ in $\un{u}_i$ is equal to $(01\dots 11)$, where $i$ is the
maximal index with $\un{u}_i \ne \emptyset$. Now any expression for
$w_A$ starting in $s_{a-1} \dots s_1$ has to contain a subsequence
to the right
of the form $s_{a-1} \dots s_2$ (think about what happens to $2
\in \{1, \dots, a + n + b\}$). Continuing in this way we see that the
restriction of ${\un e}$ to each $\un{u}_j$ has the form
\[
(01 \dots 1) \dots (01 \dots 1) (01 \dots 1)
\]
(with the same bracketing as in the definition of each $\un{u}_i$).
Similar arguments apply to each $\un{v}_i$ and the result follows.
\end{proof}

\begin{lem}
There is a unique subexpression ${\un e}$ of $\un{w}$ such that
$\un{w}^{\un e} = w_{A \cup M \cup B}$ and ${\un e}$ has defect zero.
\end{lem}

\begin{proof}
  By the previous lemma we must have $e_j = 0$ (resp. 1) if $s_{i_j}
  \in \{ s_a, s_{a+n}\}$ (resp. $s_{i_j} \in A \cup B$). Because each
  $e_j$ with $s_{i_j}\in \{ s_a, s_{a+n}\}$ is $U0$ and because $W_{A \cup B}$ and $W_M$ commute we
  only have to understand subexpressions $\un{e}'$ of
\[
\un{w}' = \un{w}_m \un{w}_{m-1} \dots \un{w}_1 \un{w}_M
\]
of defect $-(a+b) = -\sum_{i = 1}^m \ell(w_i)$ such that
$(\un{w}')^{\un{e}'} = w_M$. Now $\ell(\un{w}') = \ell(w_M) + a + b$
and hence any subexpression $\un{e}'$ of $\un{w}'$ with
$(\un{w}')^{\un{e}'} = w_M$ has at most $a + b$ zeroes. Moreover, if
$\un{e}'$ has defect $-a -b$ then $\un{e}'$ must have exactly $a + b$ zeroes,
all of which have to be $D0$.
Now, using that $\un{w}_M$ is reduced, the only subexpression of $\un{w}'$ fulfilling these requirements is
\[
(0 \dots 0)(0 \dots 0) \dots (0 \dots 0)(1 \dots 1). \qedhere
\]
\end{proof}

\begin{proof}[Proof of Theorem \ref{thm:main2}]
  We conclude from the previous two lemmas and their proofs that the
  unique defect zero subexpression ${\un e}$ of $\un{w}$ with
  $\un{w}^{\un e} = w_Aw_Bw_M$ is
  \[ {\un e} = {\un e}_m{\un f}_m{\un g}_m \dots {\un e}_2{\un
    f}_2{\un g}_2{\un e}_1{\un f}_1{\un g}_1{\un e}_0
  \]
  where ${\un e}_0$ (resp. ${\un e}_i$) is a subexpression of
  $\un{w}_M$ (resp. $\un{w}_i$) given by 
  \begin{gather*}
    {\un e} _0 = (U1, U1, \dots, U1) \quad ( \text{resp.} \quad
    {\un e}_i = (D0, D0, \dots, D0) )
  \end{gather*}
  and ${\un f_i}$ (resp. ${\un g_i}$) is a subexpression of $\un{u}_i$
  (resp. $\un{v}_i$) given by
  \[
  (U0, U1, \dots, U1)(U0, U1, \dots, U1) \dots (U0, U1,
  \dots, U1).
  \]
  (we use the same bracketing as in the definition of $\un{u}_i$ and
  $\un{v}_i$).

Hence the intersection form of $\un{w}$ at $w_{A \cup M \cup B}$ for
degree $d = 0$ is indeed a $1 \times 1$ matrix.
  Applying \cite[Theorem 5.1]{HW} its unique entry is given by the coefficient of
  $\delta_{w_{A \cup M \cup B}} =
  \delta_{w_A}\delta_{w_M}\delta_{w_B}$ in
\footnote{
Actually, as noted in Remark \ref{rem:opp}, here we use a
``right to left'' convention, rather than the ``left to right''
convention of \cite{HW}. One can check that \cite[Theorem 5.1]{HW}
holds in either
convention. Alternatively one can proceed as follows.
Let $\un{w}^r = s_{i_\ell} \dots
s_{i_1}$ denote the reversed sequence, and let $\iota : NH \to NH$
denote the anti-involution with $\iota(f) = f$ for $f \in R$ and
$\iota(\delta_x) = \delta_{x^{-1}}$ for $x \in W$. Then \cite[Theorem
5.1]{HW} implies that the intersection form of $\un{w}^r$ at $w_{A \cup M \cup B}$ is
the $1 \times 1$ matrix given by the coefficient of
$\delta_{w_A}\delta_{w_M}\delta_{w_B}$ in $\iota(E)$. This implies the
statement because the intersection forms of $\un{w}$ and $\un{w}^r$ at
$w_{A \cup M \cup B}$ agree.}
  \[
  E := E_mF_mG_m \dots E_2F_2G_2E_1F_1G_1E_0
  \]
  where $E_0 = \delta_{w_{M}}$, $E_i = \delta_{w_i}$ for $1 \le i \le m$ and the
  $F_i, G_i$ are given by:
  \begin{align*}
    F_1 &= ( \a_a \delta_{a-1}\dots \delta_{a-a_1+1} ) \dots (\a_a \delta_{a-1}) (\a_a) \\
    F_2 &= (\a_a \delta_{a-1}\dots \delta_{a-a_1-a_2+1}) \dots (\a_a \delta_{a-1}
    \dots \delta_{a-a_1-1})(\a_a\delta_{a-1} \dots \delta_{a-a_1})
    \\
    & \quad \vdots \\
    F_m &= (\a_a\delta_{a-1} \dots \delta_1) \dots (\a_a\delta_{a-1} \dots
    \delta_{a-a_1-\dots-a_{m-1} - 1})
    (\a_a\delta_{a-1}  \dots \delta_{a-a_1 - \dots - a_{m-1}} ) \\
    G_1 &= (\a_{a+n}\delta_{a+n+1} \dots \delta_{a+n+b_1-1}
    )  \dots (\a_{a+n}\delta_{a+n+1}) (\a_{a+n})\\
    & \quad \vdots \\
    G_m &= (\a_{a+n}\delta_{a+n+1} \dots \delta_{a + n + b - 1}) \dots (\a_{a +
      n}\delta_{a+n+1} \dots \delta_{a + n + b_1 + \dots + b_{m-1}}) .
  \end{align*}
 
 In $NH$ we can write $E = \sum_{y \in W_{A \cup M \cup B}} f_y
 \delta_y$. After noting that
\[
\deg E = 2(-\ell(w_A) - \ell(w_B) + a + b - \sum \ell(w_i) - \ell(w_M)) =
-2\ell(w_{A \cup M \cup B})
\]
we see that in fact $E = \kappa' \delta_{w_{A \cup M \cup B}}$
  for some $\kappa' \in \ZM$. In particular, whenever we apply a nil Hecke
  relation $f \delta_i = \delta_i s_i(f) + \partial_i(f)$ with $s_i \in A \cup B$
  to reduce $E$ the term involving $\partial_i(f)$ does not
  contribute. (It would lead to a term which is zero for degree
  reasons.) Hence
  we can write
  \begin{gather*}
    E = \delta_{w_A}\delta_{w_B}(\delta_{w_m}\g_m\d_m) \dots (\delta_{w_2}
    \g_2\d_2)(\delta_{w_1} \g_1\d_1) \delta_{w_M}
  \end{gather*}
  where each $\g_i$ (resp. $\d_i$) is a product of $a_i$ (resp. $b_i$)
  roots of the form $x_k - x_{a+1}$ with $k < a+1$ (resp. $x_{a+n} -
  x_k$ for $k > a + n$). Hence we have
  \begin{gather*}
    E = \delta_{w_Aw_B}(\delta_{w_m} (-x_{a+1})^{a_m}x_{a+n}^{b_m})
  \dots (\delta_{w_2} (-x_{a+1})^{a_2}x_{a+n}^{b_2}) 
(\delta_{w_1} (-x_{a+1})^{a_1}x_{a+n}^{b_1}) \delta_{w_M}\\
 = (-1)^a \kappa \cdot \delta_{w_{A \cup M \cup B}}
  \end{gather*}
where the first (resp. second) equality follows from Lemma
\ref{lem:D1} (resp. \ref{lem:D2}) below. The theorem follows.
\end{proof}

\begin{lem} \label{lem:D1}
  Let $w_1, \dots, w_m \in W_M$ and $\z_1, \dots, \z_m \in R$. Assume
  that for some $1 \le i \le m$ we can write $\z_i = \z_i^M\z_i'$ for some
    $W_M$-invariant $\z_i^M$ of positive degree, and that $\sum \deg \z_i = \sum \ell(w_i)$.
Then $\delta_{w_m} \z_m \dots \delta_{w_1} \z_1 \delta_{w_M} = 0$.
\end{lem}

\begin{proof}
  We have $\delta_{w_m} \z_m \dots \delta_{w_i}\z_i{'}  \dots \delta_{w_1} \z_1
  \delta_{w_M} \in \bigoplus_{ y \in W_M} R \delta_y$, and hence
  \begin{gather*}
    \delta_{w_m} \z_m \dots \delta_{w_i}\z_i{'}  \dots \delta_{w_1} \z_1 \delta_{w_M} = 0
  \end{gather*}
because it is of degree $< -2\ell(w_M)$. As $\z_i^M$ is $W_M$-invariant:
\[
0 = \z_i^M(\delta_{w_m} \z_m \dots \delta_{w_i}\z_i{'}  \dots \delta_{w_1} \z_1
\delta_{w_M}) = \delta_{w_m} \z_m \dots \delta_{w_1} \z_1 \delta_{w_M}. \qedhere
\]
\end{proof}

\begin{lem} \label{lem:D2} With $w_i, a_i, b_i, \kappa$ as above 
  we have:
\[ (\delta_{w_m} x_{a+1}^{a_m}x_{a+n}^{b_m})( \delta_{w_{m-1}}
  x_{a+1}^{a_{m-1}}x_{a+n}^{b_{m-1}}) \dots (\delta_{w_1}
  x_{a+1}^{a_1}x_{a+n}^{b_1})\delta_{w_M} = \kappa \cdot \delta_{w_M}.\]
\end{lem}
\begin{proof}It is well known that $\delta_i \mapsto \partial_i, f \mapsto (f \cdot)$ makes $R$ into an $NH$-module. 
  In $NH$ we can write
\begin{gather*} \label{eq:exp}
(\delta_{w_m} x_{a+1}^{a_m}x_{a+n}^{b_m})( \delta_{w_{m-1}}
  x_{a+1}^{a_{m-1}}x_{a+n}^{b_{m-1}}) \dots (\delta_{w_1}
  x_{a+1}^{a_1}x_{a+n}^{b_1})  = K + \sum_{\id \ne w \in W_M} f_w\delta_w
\end{gather*}
where $K \in \ZM$ for degree reasons.
By applying this identity to $1 \in R$ we deduce that $K = \kappa$. The
lemma now follows because if $w \in W_M$ then $\delta_w \delta_{w_M} = 0$
unless $w = \id$.
\end{proof}

\section{(Counter)-examples} \label{sec:ex}

We use the notation of \S \ref{sec:Schubert} and write $\dd_{12} :=
\dd_1\dd_2$, $X_1 := X_{s_1}, X_{12} := X_{s_1s_2}$ etc.

\subsection{$n < 4$:} One checks easily using \eqref{eq:dem} and 
\eqref{eq:Chevalley} that for $n = 2,3$ one can only
obtain $\kappa = \pm 1$.

\subsection{$n = 4$:} \label{sec:ex4}
Using \eqref{eq:dem} and \eqref{eq:Chevalley} we see that in $C$ we have
\begin{equation}
  \label{eq:schub}
X_1 = x_1 \qquad \text{and} \qquad  X_3 = -x_4.  
\end{equation}
Consider the (degree zero) operator $F : C \to C$ given by
\[
F : h \mapsto \dd_{23}(x_4^2(\dd_1(x_1h))).
\]
Using \eqref{eq:dem} and \eqref{eq:Chevalley} one checks that $F$
preserves the submodule $\ZM X_1 \oplus \ZM X_3$ and in the basis
$X_1, X_3$ is given by 
\[
F = \left ( \begin{array}{cc} 1 & 1 \\ 1 & 0 \end{array} \right ).
\]
This matrix determines the Fibonacci recursion! Hence for $i \ge 1$ we have
\[
F^i(x_1) = F_{i+1}X_{1} + F_iX_{3}
\]
where $F_1 = 1$, $F_2 = 1$, $F_3 =
2$, $F_4 = 3$ etc. denote the Fibonacci numbers. In particular,
\[
\dd_1(F^i(x_1)) = F_{i+1}.
\]
We conclude from the main theorem that any prime dividing the Fibonacci number $F_{i+1}$
occurs as torsion in $\SL_{3i+5}$.
By Carmichael's theorem \cite{Car}
 the first $n \gg 1$ Fibonacci numbers have at least
$n$ distinct prime factors. By the prime number theorem we conclude that
the torsion in $\SL_n$ grows at least as fast as some constant times $n
\log n$. Hence no linear bound is sufficient for Lusztig's
conjecture.

It is a well-known conjecture that infinitely many
Fibonacci numbers are prime. By the above results, this conjecture
would immediately imply that the torsion in $\SL_n$ grows
exponentially in $n$. Unfortunately, little seems to be 
known about the rate of growth of prime factors of Fibonacci numbers.

In the appendix we work with different operators in order to establish
exponential growth of torsion. If $U_l$ (resp. $U_u$) denotes the operator $h
\mapsto \dd_{21}(x_1^2(\dd_1(x_1h)))$ (resp. $h \mapsto
\dd_{23}(x_4^2(\dd_3(x_4h)))$) then $U_l$ and $U_u$ preserve the
submodule $\ZM X_1 \oplus \ZM X_3$ and in the basis $X_1, X_3$ are given by 
\[
U_l = \left ( \begin{array}{cc} 1 & 0 \\ 1 & 1 \end{array} \right ),
\qquad U_u = \left ( \begin{array}{cc} -1 & -1 \\ 0 & -1 \end{array} \right ).
\]
It follows from our main theorem that any prime dividing any matrix coefficient of
any word of length $\ell$ in the generators $\mattwo 1101$ and $\mattwo 1011
$ occurs as torsion in $\SL_{3\ell + 5}$. Indeed, given
  any word $\omega_1 \omega_2 \dots \omega_r$ in the
  operators $U_l$ and $U_u$ we may obtain all four coefficients (up to
  sign) of the corresponding product of the matrices $\mattwo 1101$ and $\mattwo 1011
$ as $\partial_{i} (\omega_1 ( \dots (\omega_r(x_{j})) \dots ))$
for $i \in \{1, 3\}$ and $j \in \{ 1, 4 \}$ (use \eqref{eq:schub}
and the fact that $\partial_1(X_1) = 1 = \partial_3(X_3)$).

\subsection{$n = 5$:} \label{sec:ex5}
In the following table we list some examples of $p$ torsion in
$\SL_N$ found using $n = 5$. The entries in the list were
found by random computer searches:
\[
\begin{array}{cccccccccccccccccc}
N &  14 & 17 & 20 & 22 &  25 & 30  & 40      & 50         & 70  & 100\\
p &    3 &   7 & 13 & 23 &  53  & 197& 2\;237 & 34\;183 & 4\;060\;219
& 470\;858\;183
\end{array}
\]
(These entries were found as follows. Consider the
eight degree zero operators:
\[
\dd_{4321}x_1^4, \dd_{321}x_1^3, \dd_{21}x_1^2, \dd_1x_1,
\dd_{1234}x_5^4, \dd_{234}x_5^3, \dd_{34}x_5^2, \dd_4x_5.
\]
It is not difficult to calculate the matrices of these operators
acting on any homogenous component of $C$ in the Schubert
basis. The above entries were obtained as prime factors of
coefficients obtained by repeated application of these operators
to $x_1^3$ and $x_1^2x_5 \in C^6$.)

\section{Lusztig conjecture} \label{sec:lusztig}

This section consists of connections and complements to
\cite{Soe}, with which we assume the reader is familiar. In keeping
with the setting of this paper, we work with $G = \SL_n$ throughout, however analogous
statements are true (with the same proofs) for any connected reductive
group.


As in \S\ref{sec:SM} we assume in this and the following
  section that $p > n$.
Let $\OC$ denote the ``regular subquotient around the Steinberg
weight'' as defined in \cite[\S 2.3]{Soe}. We denote by $\Delta(x),
P(x)$ the standard and projective objects in $\OC$ and by $\theta_s :
\OC \to \OC$
 for $s \in S$ the translation functor \cite[\S 2.5]{Soe}. Let $[\OC]$
denote the Grothendieck group of $\OC$ and
\[
\a : [\OC] \simto \ZM[W]
\]
the isomorphism with $\a([\Delta(x)]) = x$ for all $x \in W$ ($\a$ is 
denoted $A$ in \cite[\S 2.10]{Soe}). As observed in \cite{Soe},
Lusztig's conjecture implies that
\[
\a(P(x)) = {\un{H}_x}_{|v = 1} \qquad \text{for all $x \in W$.}
\]

\begin{remark}
  This observation should be compared with a much earlier theorem of
  Jantzen \cite[Anhang, Corollar]{JaMo} matching multiplicities of
  simple modules in Weyl modules in sufficiently large characteristic $p$ and
  multiplicities of simple modules in Verma modules in characteristic
  $0$. This observation, together with Jantzen's
  calculations in rank 2, were the main ingredients that led to the
  formulation of the Lusztig conjecture.
\end{remark}

\begin{prop}
We have  $\a(P(x)) = {\p\un{H}_x}_{|v = 1}$. In particular, if
$\p\un{H}_x \ne \un{H}_x$ with $p > n$ then Lusztig's conjecture fails for $\SL_n$ in
characteristic $p$.
\end{prop}

\begin{remark}
 Recall that $\p a_{y,x}, \p h_{y,x} \in \ZM_{\ge 0}[v^{\pm 1}]$ (see \S \ref{sec:SB}). In particular:
 \[ \p\un{H}_x = \un{H}_x \Leftrightarrow {\p\un{H}_x}_{|v = 1} =
 {\un{H}_x}_{|v = 1}. \]
\end{remark}

\begin{proof} 
  Let $p\OC$ denote the full subcategory of projective objects in
  $\OC$, and $[p\OC]$ its split Grothendieck group. Because $\OC$ has
  finite homological dimension, the map $[p\OC] \to [\OC]$ induced by
  the inclusion is an isomorphism. Recall the commutative diagram
\begin{gather*}
  \begin{tikzpicture}[xscale=2,yscale=1.5]
    \node (o) at (0,1) {$[p\OC]$};
    \node (fc) at (1,1) {$[f\CC]$};
\node (b1) at (0,0) {$\ZM S_n$};
\node at (0.5,0) {$=$};
\node (b2) at (1,0) {$\ZM S_n$};
\draw[->] (o) to node[above] {$\mathbb{V}$} (fc);
\draw[->] (o) to node[left] {$\a$} (b1);
\draw[->] (fc) to node[left] {$\b'$} (b2);
  \end{tikzpicture}
\end{gather*}
from \cite[\S 2.11]{Soe} (see \S\ref{sec:SM} for the definition of
$f\CC$).

We claim that $\b'$ above agrees with the $\b$
defined in \S\ref{sec:SM}. If $\un{w} = st
\dots u$ we have, by \cite[Theorem 2.6.2]{Soe}:
\[
\mathbb{V}(\theta_u \dots \theta_t\theta_s M_{\id}) \cong fD_{\un{w}}.
\]
Thus, by \cite[\S 2.5 and \S 2.10]{Soe}:
\[
\b'([fD_{\un{w}}]) = \a( \theta_u \dots \theta_t\theta_s M_{\id}) =
(1+s)(1+t) \dots (1+u).
\]
By the commutativity of the diagram
in \S\ref{sec:SM}:
\[
\b([fD_{\un{w}}]) = \ch([B_sB_t \dots B_u])_{|v=1} = (\un{H}_s \un{H}_t \dots
\un{H}_u)_{|v=1} = (1+s)(1+t) \dots (1+u).
\]
Hence $\b = \b'$ as claimed, as $[f\CC]$ is generated by
$[fD_{\un{w}}]$ over all expressions $\un{w}$.

Now we are done: by \cite[Theorem 2.8.2]{Soe} we have
$\mathbb{V}P(x) = fD_x$ and 
the proposition follows from \eqref{eq:beta}.
\end{proof}

\section{James conjecture} \label{sec:james}

In this section we explain why the results of the previous section
yield counter-examples to the James conjecture \cite{James} on the
decomposition numbers of Schur algebras and the symmetric group.

Fix positive integers $N$ and $r$. Let
$\Lambda^+(N,r)$ denote the set of partitions of $r$ into at most $N$
parts; that is, sequences $(\lambda_1, \dots, \lambda_N)$ such that
$\lambda_1 \ge \dots \ge \lambda_N \ge 0$ with $r = \sum
\lambda_i$.  Then $\Lambda^+(N,r)$ is a partially ordered set with
respect to the dominance order $\le$.

Let $S(N,r)$ denote the Schur algebra over $\ZM$ (see e.g. \cite{Green}). Its category of representations is
equivalent to the category of polynomial representations of the group
scheme $GL_N$ of fixed degree $r$. Fix a field $\Bbbk$ of characteristic $p > N$ and let
$S_\Bbbk(N,r)$ denote the Schur algebra over $\Bbbk$.  The category $\Rep S_\Bbbk(N,r)$ of finitely
generated $S_\Bbbk (N,r)$-modules is a highest weight category with simple
modules indexed by $\Lambda^+(N,r)$. Given
$\lambda \in \L^+(N,r)$ we denote by $L(\l)$ (resp. $\Delta(\l)$,
$\nabla(\l)$, $P(\l)$, $T(\l)$) the simple
(resp. standard, costandard, indecomposable projective, indecomposable
tilting) module indexed by $\l$.

Let $S_q(N,r)$ denote the $q$-Schur algebra and $S_\e(N,r)$ its
specialisation at a fixed primitive $p^{th}$-root of unity $\e \in
\CM$ (see e.g. \cite{Donkin}). Then the category of finitely generated $S_\e(N,r)$-modules
is highest weight. As above we write $L_\e(\lambda)$
(resp. $\Delta_\e(\lambda)$ etc.) for the simple (resp. standard etc.) module corresponding to
$\l \in \L^+(N,r)$. Given a module $M$ for $S_\e(N,r)$ we may choose a
stable $\ZM[\e]$-lattice and reduce to obtain a module over
$S_\Bbbk(N,r)$. In this way we obtain the 
decomposition map on Grothendieck groups
\[
d : [\Rep S_\e(N,r)] \to [\Rep S_\Bbbk(N,r)].
\]
One has $d([\nabla_\e(\l)]) = [\nabla(\l)]$. The James conjecture
\cite{James} predicts that
\begin{equation} \label{eq:james}
d([L_\e(\lambda)]) = [L(\lambda)]
\end{equation} if $p > \sqrt{r}$.\footnote{A stronger version requires
that $p$ be larger than the weight of $\lambda$. It reduces to
the condition $p > \sqrt{r}$ for the principal block, which
will be the only case considered below.}

Let $\rho = (N-1, \dots, 1, 0) \in \Lambda^+( N, {N \choose 2})$ and let
$\st := (p-1)\rho$ denote the ``Steinberg weight''. Let $S_N$ denote
the symmetric group of $N$ letters, acting by permutation on
$\ZM^N$. Lusztig's quantum character formula gives (see \eqref{eq:pcan} for notation)
\begin{equation} \label{qlusztig}
[\nabla_\e(\st + x\rho) : L_\e(\st + y\rho)] = h_{x,y}(1).
\end{equation}
as modules for $S_\e(N, p{N \choose 2})$.
(Actually, Lusztig's quantum character formula gives the multiplicity for the
quantum group of $\mathfrak{sl}_N$ specialised at $\e$ in terms of an
affine Kazhdan-Lusztig polynomial. The translation of his formula to
yield the above multiplicity is standard but technical. Alternatively,
one can appeal to \cite{AJS} and the $p \gg 0$ version of
\eqref{eq:SoergelMult} below.)

Similarly, \cite[Theorem 1.2.2]{Soe} gives (again see \eqref{eq:pcan} for notation)
\begin{equation} \label{eq:SoergelMult}
[\nabla(\st + x\rho) : L(\st + y\rho)] = \p h_{x,y}(1)
\end{equation}
as modules for $S_\Bbbk(N, p{N \choose 2})$. 
(Actually, Soergel's result gives this multiplicity for rational
modules for $\SL_N(\Bbbk)$. The translation to $GL_N(\Bbbk)$ and hence to
modules for the Schur algebra is standard.)

We conclude that whenever
$\p\un{H}_x \ne \un{H}_x$ for some $x$ the characters of 
the simple modules for $S_\e(N, p{N \choose 2})$ and $S_\Bbbk(N, p{N
  \choose 2})$ are different, because the simple and costandard modules
both give bases for the Grothendieck group. In particular, there
exists $\l$ such that $d([L_\e(\lambda)]) \ne [L(\lambda)]$.
Hence any $p$ appearing on the table in Section \ref{sec:ex} with $p > { N \choose 2}$
contradicts the James conjecture for  $S(N,
p{ N \choose 2})$.

\begin{remark}
A straightforward computation in $[\Rep S_\e(N,r)]$ and $[\Rep S_\Bbbk(N,r)]$
shows that 
\[
d([L_\e(\st + x\rho)]) = \sum a_{x,y}(1) [L(\st + y\rho)]
\]
and so the $a_{x,y}$ evaluated at 1 give part of
James's ``adjustment matrix''.
\end{remark}

However, amongst weights of the form $\st + x\rho$ for $x \in S_N$ 
only $\st + w_0 \rho$ is $p$-restricted ($w_0$ denotes the longest
element of $S_N$). Hence the above non-trivial decomposition numbers
are invisible to the symmetric group, as all simples corresponding
to non $p$-restricted weights are killed by the
Schur functor.

To get counter-examples in the symmetric group we can use the Ringel
self-duality of the Schur algebra and modular category $\OC$. (I thank Joe Chuang for explaining this to me.)
Given any $N' \ge N$ we have an obvious embedding $\L^+(N,
r) \hookrightarrow \L^+(N', r)$ obtained by appending 0's to the
partition. There is a quotient functor $f : \Rep S(N',r) \to
\Rep S(N,r)$ which preserves simple, standard, costandard modules
and indecomposable tilting modules corresponding
to $\l$ in $\L^+(N, r) \subset \L^+(N',r)$ (see \cite[\S 6.5]{Green}
and \cite[\S A4.5]{Donkin}).

Now consider a variant of the subquotient around the Steinberg weight
discussed in the previous section. Consider the Serre subquotient $\OC := \AC / \NC$ of $\Rep S_\Bbbk(N,p{N
  \choose 2})$ where $\AC$ 
is the Serre subcategory generated by simple modules $L(\lambda)$
such that $\lambda \le st + \rho$ and $\lambda$ lies in the same
block as $st + \rho$, and $\NC$ denotes the Serre subcategory
generated by those simples $L(\lambda) \in \AC$ which are not of the
form $L(\st + x\rho)$ for some $x \in S_N$. The definition of $\OC_\e$
is obtained by replacing $L(\lambda)$ by $L_\e(\lambda)$ in the
definition of $\OC$.

Let us denote the images of $L(\st + x\rho)$, $\Delta(\st + x\rho)$,
etc. in $\OC$ as $L(x)$, $\Delta(x)$, etc. and similarly for
$\OC_\e$. Then $\OC$ is a highest weight category with simple,
standard, etc. and tilting objects $L(x)$, $\Delta(x)$, etc., and similarly for $\OC_\e$.
It is known that both categories are Ringel self-dual. (The proof of this fact seems not to
be explicit in the literature. However a proof may be obtained by
adapting ideas of \cite{BBM}. One shows that one has a braid group
action on $D^b(\OC)$ (resp. $D^b(\OC_\e)$) and a lift of the longest
element interchanges a projective and tilting generator.) Applying
BGG reciprocity then Ringel self-duality for $\OC$ and $\OC_\e$ we obtain:
\begin{gather*} 
 h_{x,y}(1) = [\nabla_\e(x) : L_\e(y)] = (P_\e(y):\Delta_\e(x)) =
(T_\e(yw_0):\nabla_\e(xw_0)),  \\
 \p h_{x,y}(1) = [\nabla (x) : L (y)] = (P(y):\Delta(x)) =
(T (yw_0):\nabla (xw_0)).
\end{gather*}
Applying Ringel self-duality of $S(p{N \choose 2},
p{N\choose 2})$ \cite{DonkinTilting} and of $S_\e(p{N \choose 2},
p{N\choose 2})$ \cite{Donkin} 
we have:
\begin{gather*}
h_{x,y}(1) =
(T_\e(st+yw_0\rho):\nabla_\e(st+xw_0\rho)) =
(P_\e((st+yw_0\rho)'):\Delta_\e((st+xw_0\rho)')), \\
\p h_{x,y}(1) =
(T(st+yw_0\rho):\nabla (st+xw_0\rho)) =
(P((st+yw_0\rho)'):\Delta ((st+xw_0\rho)')).
\end{gather*}
Finally, again by BGG reciprocity
\begin{gather*}
h_{x,y}(1) = 
[\nabla_\e((st+xw_0\rho)') : L_\e((st + yw_0\rho)') ],\\
\p h_{x,y}(1) = 
[\nabla((st+xw_0\rho)') : L((st + yw_0\rho)') ].
\end{gather*}
The partitions $(\st + xw_0\rho)'$ and $(\st + yw_0\rho)'$ are
$p$-restricted. Hence, after applying the Schur functor the first (resp. second) number can be
interpreted as a decomposition number for the Hecke algebra
specialised at $\e$ (resp. the symmetric group in characteristic
$p$). It follows that the results of the previous section also 
produce counter-examples for the symmetric group.

\begin{remark}
  Consulting the table of counter-examples in Section \ref{sec:ex} we
  see that the
  smallest counter-example produced by the above methods occurs in $S_{N'}$ with
\[ N' = p{N \choose 2} = 2237{40 \choose 2} = 1\;744\;860.\]
The size of this number is a relic of our
method (in particular the fact that we cannot say anything about $p$-restricted
weights). It is an important question as to where the first counter-examples
occur.
\end{remark}

\appendix
\section{Exponential Growth of Torsion} \label{app}

\centerline{ \emph{by Alex Kontorovich, Peter J. McNamara and Geordie Williamson.}
\footnote{Kontorovich is partially supported by
an NSF CAREER grant DMS-1455705  and an Alfred P. Sloan Research
Fellowship.}
}
\hspace{1cm}

\subsection{Statement of the theorem}

Let 
\begin{equation}\label{eq:GamDef}
\G\ :=\left \langle \mattwo 1101,\mattwo 1011 \right \rangle ^{+}
\end{equation}
be the sub-{\it semi}-group of $\SL(2,\Z)$ generated (freely) by the matrices displayed. 
For a matrix $\g\in\G$, let $\ell(\g)$ be its wordlength in the generators of $\G$. 
In the main body of the paper, the third-named author proves that any prime $p$ dividing any coefficient $\g_{ij}$ of any matrix 
$$
\g=\mattwo{\g_{11}}{\g_{12}}{\g_{21}}{\g_{22}}\ \in\ \G
$$
occurs as torsion in $\SL_{3\ell+5}$, where
$\ell=\ell(\g)$ is the wordlength (see \S \ref{sec:ex4}). The purpose of this appendix is to
show the existence of  exponentially large (relative to wordlength)
prime divisors of matrix coefficients in $\G$, thus giving
exponentially large counterexamples to the expected
  bounds in Lusztig's conjecture. 
\\

In fact, the stated purpose can be accomplished by an almost\footnote{For the application to be immediate, $\G$ would need to be a Zariski-dense {\it group} and not just a semi-group; minor modifications are needed to handle this case.} direct application of the Affine Sieve \cite{BourgainGamburdSarnak2006, BourgainGamburdSarnak2010, SalehiSarnak2013}; see also, e.g., \cite{Kontorovich2014}. It turns out that one can do much more using recent 
progress
 on ``local-global'' problems in ``thin orbits'' (see, e.g., the discussion in \cite{Kontorovich2013}); namely, one can produce not just prime divisors but actual primes in the entries of $\G$, and moreover give explicit estimates for their exponential growth rates (which are far superior compared to those which would come from an Affine Sieve analysis). Our main result is the following

\begin{thm}\label{thm:appmain}
There are absolute constants $\tau>0$ and $c>1$ so that, for all $L$
large, there exists $\g\in\G$ of wordlength $\ell(\g)\le L$ and
top-left entry $\g_{11}=p$ prime with $p>\tau c^{L}$. In fact, there are many primes arising this way:
\begin{equation}\label{eq:main}
\#\left\{
p>\tau c^{L}
\ : \
\exists\g\in\G\text{ with }\ell(\g)\le L\text{ and }\g_{11}=p
\right\}
\  \gg \
{c^{L}\over L}.
\end{equation}
The implied constant above is absolute and effective.
\end{thm}

Throughout this appendix, $p$ always denotes a prime. The notation
$f(L) \gg g(L)$ means that $g = O(f)$, i.e. 
$|g(L)| \le M |f(L)|$ for a fixed $M > 0$ and all large $L$. In this
case, $M$
is the \emph{implied constant} referred to above.


Exact estimates for $\tau$ and $c$ can be readily determined; the value
coming from our proof is
$c=(\frac{1+\sqrt{5}}{2})^{1/5}\approx1.101\dots$ and we can take $\tau =
5/7$; see \eqref{eq:cIs}.
It turns out that \thmref{thm:appmain} is a nearly immediate
consequence of recent advances towards Zaremba's conjecture on
continued fractions with bounded partial quotients. 

Given $A\ge1$, let $\G_{A}$ be the sub-semigroup:
\begin{equation}\label{eq:GAdef}
\G_{A}\ := \
\left\langle
\mattwo a110\cdot\mattwo b110
 \ :\
1\le a,b\le A
\>^{+}
.
\end{equation}
(In fact, $\Gamma_A$ is freely generated by the displayed elements.)

\begin{thm}[Bourgain-Kontorovich
  \cite{BourgainKontorovich2014}]\label{thm:BK} There exists $A_0$ and
  an absolute constant $\fc < \infty$ so that, for $A \ge A_0$ and all $N$ large,
$$
\#\{
n\le N \ : \
\exists\g\in\G_{A}\text{ with }\g_{11}=n
\}
\ = \
N
\left(
1+
O\left(
e^{-\fc\sqrt{\log N}}
\right)
\right)
,
$$
where  the implied constant and $\fc>0$  are both absolute.
\end{thm}

That is, almost all {\it integers} (not just primes) arise as top-left entries in the semigroup $\G_{A}$. 
Bourgain-Kontorovich give $A_0=50$ as an allowable value for $A$, and
this has since been reduced to $A_0=5$ \cite{FrolenkovKan2013, Huang2013}; furthermore, Hensley \cite{Hensley1996} has conjectured that $A_0=2$ is allowable, and that the error rate $O(e^{-\fc\sqrt{\log N}})$ can be replaced by $O(1/N)$.
What is most important to our application is that the error rate  is asymptotically $o(1/\log N)$. This, together with the Prime Number Theorem, 
has the following immediate
\begin{cor}
Let notation be as above and set $A=5$. Then for any fixed constant $\th<1$,
\begin{equation}\label{eq:cor}
\#\left\{
p\in(\th N, N]
\ : \
\exists\g\in\G_{A}\text{ with }\g_{11}=p
\right\}
\  = \
(1-\th)
{N\over \log N}
\bigg(1+o(1)\bigg)
,
\end{equation}
as $N\to\infty$.
\end{cor}

Equipped with this estimate, is it a simple matter to give the

\subsection{Proof of \thmref{thm:appmain}}Fix constants $c>1$ and $\tau>0$ to
be chosen later, and let $\SC_{1}$ denote the set of primes on the
left hand side of \eqref{eq:main},
$$
\SC_{1} \ := \ 
\left\{
p>\tau c^{L}
\ : \
\exists\g\in\G\text{ with }\ell(\g)\le L\text{ and }\g_{11}=p
\right\}
.
$$
We seek a lower bound on the cardinality of $\SC_{1}$.

For a parameter $A$ (which we will soon set to $A=5$) and a matrix  $\g\in\G_{A}$, let $\ell_{A}(\g)$ denote the wordlength in the generators of $\G_{A}$ 
given in \eqref{eq:GAdef}.
We make the pleasant observation that 
$$
\mattwo a110\cdot\mattwo b110
\ =\ 
\mattwo 1a01\cdot\mattwo 10b1
,
$$
and hence $\G_{A}$ is a sub-semigroup of $\G$. Moreover, if $\g\in\G_{A}\subset\G$, then the wordlengths in the two semigroups are related by
$$
\ell(\g) \ \le \ 2A\cdot\ell_{A}(\g)
,
$$
since each generator in $\G_{A}$ has wordlength at most $2A$ in the generators of $\G$.
We decrease $\SC_{1}$ to a smaller set $\SC_{2}\subset \SC_{1}$ of primes coming from the top-left entries of $\G_{A}$ instead of $\G$:
$$
\SC_{2} \ := \ 
\left\{
p>\tau c^{L}
\ : \
\exists\g\in\G_{A}\text{ with }\ell_{A}(\g)\le L/(2A)\text{ and }\g_{11}=p
\right\}
.
$$

Next we define the archimedean sup-norm 
$$
\|\g\|_{\infty} \ := \ \max(\g_{ij}),
$$  
which for $\g\in\G_{A}$ is easily seen to be the top left entry 
\begin{equation}\label{eq:sup11}
\|\g\|_{\infty}\ = \ \g_{11}
.
\end{equation}
Let 
$$
\varphi \ := \ {1+\sqrt{5}\over 2} \qquad\text{and}\qquad \overline{\varphi} \ := \ {1-\sqrt{5}\over 2}
$$
denote the eigenvalues of $\mattwos 1110$. 

For any $$\g=\prod_{i=1}^n \left ( \mattwo {a_i}110 \mattwo{b_i}110
\right ) \in\Gamma_A,$$ we have
\begin{align*}
 \|\g\|_{\infty} &= \left(\begin{matrix}
                     1 & 0
                    \end{matrix}\right) \g \left(\begin{matrix}
1 \\
0
\end{matrix}\right)  \geq \left(\begin{matrix}
                     1 & 0
                    \end{matrix}\right) \mattwo 1110 ^{2n} \left(\begin{matrix}
1 \\
0
\end{matrix}\right)  =F_{2n+1}
\end{align*}
where $F_{m}$ is the $m$-th Fibonacci number. Because $F_{2n+1} =
(\varphi^{2n+1} - \overline{\varphi}^{2n+1})/\sqrt{5}$, if we set $d := \varphi
/ \sqrt{5}$, then for all $\g\in\G_{A}$,
$$
\|\g\|_{\infty} \ \geq \ d\cdot \varphi^{2\ell_{A}(\g)}
.
$$
That is, the logarithm of the archimedean norm is controlled (up to a constant) by the wordlength. 
Define the ``archimedean'' parameter $N$ (with respect to $L$)  
by
\begin{equation}\label{eq:Nis}
N \ := \ d\cdot \varphi^{L/A}
.
\end{equation}
Replacing the wordlength condition $\ell_{A}(\g)\le L/(2A)$ in $\SC_{2}$ by the stronger restriction that $\|\g\|_{\infty}\le N$ decreases 
$\SC_{2}$ to a subset $\SC_{3}$ defined by
$$
\SC_{3} \ := \ 
\left\{
p>\tau c^{L}
\ : \
\exists\g\in\G_{A}\text{ with }\|\g\|_{\infty}\le N\text{ and }\g_{11}=p
\right\}
.
$$
Since
 $\g_{11}=p=\|\g\|_{\infty}$, the condition $\|\g\|_{\infty}\le N$ can be replaced by $p\le N$;
 hence
$$
\SC_{3} \ = \ 
\left\{
\tau c^{L}<p\le N
\ : \
\exists\g\in\G_{A}\text{ with }
\g_{11}=p
\right\}
.
$$
 
Make the choice 
\begin{equation}\label{eq:cIs}
c
\ = \ \varphi^{1/A},
\end{equation}
 which is 
$(\frac{1+\sqrt{5}}{2})^{1/5}\approx 1.101\ldots$ when $A=5$.
Then for any $\th<1$, take $\tau=\theta d$. With these choices of parameters, we see that
$$
\SC_{3} \ = \ 
\left\{
\th N < p \le N
\ : \
\exists\g\in\G_{A}\text{ with }\g_{11}=p
\right\}
.
$$
Now we are done: combining the above with \eqref{eq:Nis} and  \eqref{eq:cor}  gives
$$
\# \SC_{1}  \ \ge \   \# \SC_{3}  \ \gg \ \frac N{\log N} \ \gg \ {c^{L}\over L},
$$
as claimed in \eqref{eq:main}.
This completes the proof of \thmref{thm:appmain}. 

\def\cprime{$'$} \def\cprime{$'$} \def\cprime{$'$}

\end{document}